\documentclass{amsart}
\usepackage[utf8]{inputenc}

\usepackage[all]{xy}
\usepackage{hyperref}
\usepackage{amssymb}
\usepackage{mathdots}
\usepackage{mathtools}

\usepackage{tikz,tkz-euclide}

\hypersetup{
    colorlinks,    citecolor=blue,    filecolor=blue,    linkcolor=blue,    urlcolor=blue
}

\usepackage{booktabs}
\usepackage{mathrsfs}
\usepackage{todonotes}
\usepackage{bbm}
\usepackage[lite,bibtex-style]{amsrefs}
\usepackage{longtable}

\usepackage{MnSymbol}

\newtheorem{introthm}{Theorem}

\newtheorem{theorem}{Theorem}[section]
\newtheorem{lemma}[theorem]{Lemma}
\newtheorem{proposition}[theorem]{Proposition}

\theoremstyle{definition}
\newtheorem{notation}[theorem]{Notation}
\newtheorem{definition}[theorem]{Definition}

\newtheorem{question}[theorem]{Question}

\DeclareMathOperator{\rd}{res}
\DeclareMathOperator{\pp}{\mathbb P}
\DeclareMathOperator{\Cl}{Cl}
\DeclareMathOperator{\Eff}{Eff}

\DeclareMathOperator{\Num}{N}

\DeclareMathOperator{\Nef}{Nef}
\DeclareMathOperator{\barEff}{\overline{\hbox{\rm Eff}}}

\DeclareMathOperator{\Pic}{Pic}

\newcommand{\bQ}{\mathbb Q}
\newcommand{\bR}{\mathbb R}
\newcommand{\bZ}{\mathbb Z}

\newtheorem{remark}[theorem]{Remark}
\theoremstyle{remark}

\title{Effective cone of the blow up of
the symmetric product of a curve}
\author[A.~Laface]{Antonio Laface}
\address{
Departamento de Matem\'atica,
Universidad de Concepci\'on,
Casilla 160-C,
Concepci\'on, Chile}
\email{alaface@udec.cl}

\author[L.~Ugaglia]{Luca Ugaglia}
\address{
Dipartimento di Matematica e Informatica,
Universit\`a degli studi di Palermo,
Via Archirafi 34,
90123 Palermo, Italy}
\email{luca.ugaglia@unipa.it}

\thanks{Both authors have been partially supported by Proyecto
FONDECYT Regular n. 1190777}
\date{\today}

\begin{document}
\begin{abstract}
Let $C$ be a smooth curve of genus $g \geq 1$
and let $C^{(2)}$ be its second symmetric 
product. In this note we prove that if $C$
is very general, then the blow-up of 
$C^{(2)}$ at a very general point has 
non-polyhedral pseudo-effective cone.
The strategy is to consider first the case of 
hyperelliptic curves and then to show 
that having polyhedral pseudo-effective cone 
is a closed property for families of surfaces.
\end{abstract}

\maketitle


\section*{Introduction}
The study of the effective cone of the 
blow up $\tilde S$ of a  projective surface $S$ at
a smooth point $x\in S$ is connected with
the calculation of Seshadri constants.
Deciding when the (pseudo)effective cone
of $\tilde S$ is polyhedral is an open
problem even when $S$ is a toric surface.
For instance, if the effective cone of
the blow up of the weighted projective plane
$\pp(a,b,c)$ at a general point is not closed,
then Nagata's Conjecture holds for $abc$
points in $\pp^2$, see~\cite{CK} and 
\cites{hkl,gk,ggk,ggk1} for recent results
on blow ups of weighted projective planes.
In~\cite{CLTU} it has
been shown that there exist toric
surfaces whose blow up at a 
general point has non-polyhedral
pseudo-effective cone. This result allows one
to deduce that the pseudo-effective
cone of the Grothendieck-Knudsen moduli space $\bar M_{0,n}$ is not polyhedral
for $n\geq 10$.

In this paper we focus on the second symmetric
product $C^{(2)}$ of a positive genus 
curve $C$. In general it is
not known if the effective cone of these
surfaces is open. This would be true if 
the Nagata Conjecture holds, as shown
in~\cite{CiKo}.
Our interest is in the blow up $\tilde C^{(2)}$
at a very general point $p\oplus q\in C^{(2)}$.
\begin{introthm}
\label{thm:1}
Let $C$ be a very general
curve of genus $g\geq 1$. 
Then the blow-up of the symmetric
product $C^{(2)}$ at a very general point
has non-polyhedral pseudo-effective cone.
\end{introthm}
In order to prove the theorem we first show, in
Proposition~\ref{prop:eff3}, that
having polyhedral pseudo-effective cone 
is a closed property for families of surfaces
and then we prove the following.
\begin{introthm}
\label{thm:2}
Let $C$ be a genus $g\geq 1$ hyperelliptic 
curve with hyperelliptic involution 
$\sigma$, let $p\in C$ and let $\tilde C^{(2)}$
be the blow-up of $C^{(2)}$ at $p\oplus \sigma(p)$.
If the class of $\sigma(p) - p$ is non-torsion
in ${\rm Pic}^0(C)$ then $\overline{\rm Eff}(\tilde C^{(2)})$ is  
non-polyhedral.
\end{introthm}

When $C$ is an elliptic curve, its symmetric 
product is the {\em Atyiah surface}. In this case
in~\cite{Ga} it has been proved that if 
$q-p$ is non-torsion, then $\tilde C^{(2)}$ 
contains infinitely many negative 
curves. Therefore the pseudo-effective
cone of $\tilde C^{(2)}$ is not
polyhedral, and in~\cite{McG} it is
proved that the classes of the above
mentioned curves (together with other two classes)
indeed generate the pseudo-effective cone.

Our proof of Theorem~\ref{thm:2} focuses
on the quotient surface $\tilde X$ by the
action of the hyperelliptic involution on 
both factors. 
We show that there is an irreducible curve
$B$ on $\tilde X$ having self intersection
$B^2 = 0$, whose class spans an extremal
ray of the pseudo-effective cone of 
$\tilde X$, so that
the latter cannot be polyhedral
by~\cite{CLTU}*{Proposition 2.3}. 
We then apply Proposition~\ref{prop:eff1}
to the double cover  
$\tilde C^{(2)} \to \tilde X$
to conclude that the pseudo-effective
cone of $\tilde C^{(2)}$ is not
polyhedral.


The paper is structured as follows. In 
Section~\ref{sec:pre} we recall some
definitions and we prove some preliminary
results about the effective cone of 
projective surfaces. In Section~\ref{sec:sym}
we study the symmetric product $C^{(2)}$
of a curve, with particular emphasis on the
case $C$ hyperelliptic.
Section~\ref{sec:proof} is devoted to the
proof of Theorem~\ref{thm:1} and~\ref{thm:2},
while in Section~\ref{sec:ell} we prove
some results in case $g(C) = 1$.\\

{\bf Ackowledgements.}
It is a pleasure to thank Jenia Tevelev 
for several interesting discussions on
the subject of this paper.

\section{Preliminaries}
\label{sec:pre}
Let $k$ be an algebraically closed field of arbitrary characteristic. 
We recall some definitions (see for example \cites{LazI, LazII}). 
If  $X$ is a normal projective irreducible variety over $k$, let $\Cl(X)$ be the divisor class group and let 
$\Pic(X)$ be the Picard group of $X$. As usual, we denote by $\sim$ the linear equivalence of divisors
and by $\equiv$ the numerical equivalence.
Recall that for  Cartier divisors $D_1$, $D_2$, we have
 $D_1\equiv D_2$ if and only if $D_1\cdot C=D_2\cdot C$, for any curve $C\subseteq X$. 
We let
\[
\Num^1(X):=\Pic(X)/\equiv
\]
be the {\em N\'eron-Severi group},
i.e. the group of numerical
equivalence classes of Cartier divisors on $X$. 
We denote by $\rho(X)$ the
rank of $\Num^1(X)$ and by
$\Num^1(X)_{\bR}=\Num^1(X)\otimes_{\bZ}\bR$, $\Num^1(X)_{\bQ}=\Num^1(X)\otimes_{\bZ}\bQ$. 
We define the pseudo-effective cone
 \[
 \barEff(X)\subseteq \Num^1(X)_{\bR},
 \]
as the closure of the effective cone $\Eff(X)$, i.e., the convex cone generated by numerical 
classes of effective Cartier divisors (\cite{LazII}*{Def.~2.2.25}). We let 
$\Nef(X)\subseteq \Num^1(X)_\bR$ be the
cone generated by the classes of {\em nef divisors}.  


\begin{proposition}
\label{prop:eff1}
Let $f\colon X\to Y$ be a finite surjective 
morphism of normal $\mathbb Q$-factorial projective
varieties. If $\varrho(X) = \varrho(Y)$, 
then $f_*\colon {\rm N}^1(X)_{\mathbb R}\to
{\rm N}^1(Y)_{\mathbb R}$ is an isomorphism
such that $f_*({\rm Eff}(X)) = {\rm Eff}(Y)$.
\end{proposition}
\begin{proof}
Since $Y$ is $\mathbb Q$-factorial, the image of ${\rm Pic}(Y)$ in the N\'eron-Severi group
${\rm N}^1(Y)$ has finite index. Over this
subgroup the pullback is defined and the 
projection formula gives $f_*\circ f^* = n\cdot{\rm id}$, where $n = \deg(f)$. This, together with
the hypothesis $\rho(X) = \rho(Y)$, imply that 
$f_*\colon {\rm N}^1(X)_{\mathbb R}
\to {\rm N}^1(Y)_{\mathbb R}$ is an isomorphism
whose inverse is $\frac{1}{n}f^*$.
Then one concludes by the inclusions 
\[
 f_*({\rm Eff}(X))
 \subseteq {\rm Eff}(Y)
 \text{ and }
 f^*({\rm Eff}(Y))\subseteq {\rm Eff}(X).
\]
\end{proof}

\begin{proposition}
\label{prop:eff2}
Let $X$ be a normal $\mathbb Q$-factorial 
algebraic surface with $\varrho(X)\geq 3$
and positive light cone $Q\subseteq
N^1(X)_{\mathbb R}$. Let $C_1,\dots,C_n$
be irreducible curves of $X$.
Then the following are equivalent:
\begin{enumerate}
\item
$Q\subseteq {\rm Cone}([C_1],\dots,[C_n])$;
\item
$\overline{\rm Eff}(X) = {\rm Cone}([C_1],\dots,[C_n])$.
\end{enumerate}
Moreover if $\overline{\rm Eff}(X)$ is
polyhedral
then $\overline{\rm Eff}(X) = {\rm Eff}(X)$
holds and both cones are generated by 
classes of negative curves.
\end{proposition}
\begin{proof}
We prove $(1)\Rightarrow (2)$.
Let $[D]$ be a divisor class which
generates an extremal ray of the effective cone. 
Then $D^2<0$ so that the hyperplane
$D^\perp$ intersects
$Q$ along its interior.
As a consequence at least one of 
the $C_i$ satisfies $D\cdot C_i < 0$.
Thus any effective multiple of $D$ 
contains $C_i$ into its support, so 
that $[D]=[C_i]$ up to multiples.

The implication $(2)\Rightarrow (1)$
is obvious.
\end{proof}

\begin{proposition}
\label{prop:eff3}
Let $X\to B$ be a flat projective morphism
of Noetherian schemes, whose general fiber
is a normal $\mathbb Q$-factorial 
surface with Picard lattice isometric to
the one of the special fiber $X_0$ over $0\in B$.
If the general fiber has polyhedral
pseudo-effective cone, then the same holds
for the special fiber.
\end{proposition}
\begin{proof}
If the Picard rank is $\leq 2$, then the
presudoeffective cone is polyhedral and
there is nothing to prove. We then 
assume that the Picard rank is 
at least $3$. 
By Proposition~\ref{prop:eff2} the pseudo-effective
cone of the general fiber is generated by
finitely many classes of negative curves
$C_1,\dots,C_n$. By semicontinuity of cohomology
dimension, each such curve $C_i$ degenerate
to a, possibly reducible, curve of $X_0$.
Let $C_{i1},\dots,C_{ir_i}$ be the irreducible
components of the degenerate curve. We claim
that in the N\'eron-Severi space
of the special fiber $X_0$
the following inclusions of cones hold
\[
 Q
 \subseteq
 {\rm Cone}([C_i]\, :\, 1\leq i\leq n)
 \subseteq
 {\rm Cone}([C_{ij}]\, :\, 1\leq i\leq n,
 \, 1\leq j\leq r_i).
\]
Indeed, by Proposition~\ref{prop:eff2}, the first inclusion holds true 
in the N\'eron-Severi space of the general 
fiber and, by the assumption on the Picard lattice of the special fiber, it holds
as well on the N\'eron-Severi space of the 
special fiber.
The second inclusion follows by the definition of the curves $C_{ij}$.
Then, again by Proposition~\ref{prop:eff2}, 
one concludes that $\overline{\rm Eff}(X_0)
= {\rm Cone}([C_{ij}]\, :\, 1\leq i\leq n,
 \, 1\leq j\leq r_i)$.
\end{proof}

\section{Symmetric product of a curve}
\label{sec:sym}
Given a genus $g\geq 1$ curve $C$, we denote by
$C^{(2)}$ its second symmetric product, that is
the quotient of $C\times C$ by the involution $\tau$, defined by $(p,q)\mapsto (q,p)$, and  
we denote by $p\oplus q\in C^{(2)}$ the
class of $(p,q)\in C\times C$.

From now on we assume that $C$ is hyperelliptic, 
we fix a hyperelliptic involution $\sigma$ and 
we denote by $p_1,\dots,p_{2g+2} \in C$ its fixed points.
Observe that $\sigma$ induces two commuting involutions $\sigma_1,\sigma_2$ on $C\times C$, each of
which acts only on one  coordinate. 
The group $G :=
\langle\sigma_1,\sigma_2,\tau\rangle$
is isomorphic to $D_4$, with center generated 
by the composition $\sigma_1\cdot\sigma_2$, that
we still denote by $\sigma$ with abuse of notation.
We have the following diagram of degree two 
quotient morphisms
\[
 \xymatrix{
  C\times C\ar[r]\ar[d]
  & S\ar[r]\ar[d]
  & \mathbb P^1\times\mathbb P^1\ar[d]\\
  C^{(2)}\ar[r]_-{\phi}
  & X\ar[r]_-{\psi}
  & \mathbb P^2,
 }
\]
where each vertical map is the
quotient by $\tau$, the first orizontal map on each line is the 
quotient by $\sigma$, and the second
one is the quotient by $\sigma_1$.

\begin{remark}
\label{rem:tan}
Let us consider the diagonal $\Delta_+:=\{p\,\oplus\, p\mid p \in C\}$ 
and the antidiagonal $\Delta_- := \{p\oplus\sigma(p) \mid p \in C\}$ in 
$C^{(2)}$. We set $C_\pm := \phi(\Delta_\pm)\subseteq X$
and $\Gamma := \psi(C_+) = \psi(C_-)\subseteq \pp^2$. 
From the
above diagram we see that $\Gamma$ is the image of the
diagonal of $\pp^1\times\pp^1$ via the double cover
defined by $([s_0:s_1],[t_0:t_1])
\mapsto [s_0t_0:s_1t_1:s_0t_1+s_1t_0
]$, so that it is the conic $\Gamma =
V(x_3^2-4x_1x_2) \subseteq \pp^2$.

Given a point $p\in C$, consider the two curves
$\{p\}\times C$ and $C\times\{p\}$ in $C\times C$.
On $\pp^1\times\pp^1$ they are mapped to two 
lines on the two different rulings, while on 
$C^{(2)}$ they are mapped to the curve 
$C_p:=\{p\oplus q \mid q \in C\}$. 
We set $B_p:=\phi(C_p)\subseteq X$ and
$L_p:=\psi(B_p) \subseteq \pp^2$.
Observe that $B_p$ is isomorphic to $C$, while
$L_p$ is a line which is tangent to $\Gamma$ at 
the image of $p\oplus p$ (equivalently, at the image of
$\sigma(p)\oplus p$) in $\pp^2$.
We finally remark that given the curve
$C_{\sigma(p)}:=\{\sigma(p)\oplus q \mid q \in C\}$, 
we have $\phi(C_{\sigma(p)}) = 
\phi(C_p) = B_p\subseteq X$. 
\end{remark}

\begin{proposition}
\label{prop:S1}
The surface $X$ is a double cover of the plane,
branched along the union of $2g+2$ lines,
tangent to the conic $\Gamma$.
It has $\binom{2g+2}{2}$ singular points, 
namely the ordinary double points points 
$\phi(p_i\oplus p_j)$, for $1\leq i < j \leq 2g+2$.
The equation of $X$ in the weighted projective space 
$\mathbb P(1,1,1,g+1)$ is
\[
 x_4^2 + \prod_{i=1}^{2g+2}\ell_i = 0,
\]
where $\ell_1,\dots,\ell_{2g+2}\in\mathbb C[x_1,x_2,x_3]$
are defining polynomials for the $2g+2$ lines.
\end{proposition}
\begin{proof}
The ramification divisor of 
$\psi : X\to\pp^2$ consists
of the images of points 
$(p,q)\in C\times C$
such that the orbit of $(p,q)$ with 
respect to $\langle \sigma,\tau\rangle$
equals the orbit with respect to the 
whole group $G$, that is
$(\sigma(p),q)\in\{(p,q),(q,p)$,
$(\sigma(p),\sigma(q))$,$(\sigma(q),\sigma(p))\}$.
The latter condition holds if and only if 
either $p$ or $q$ is a fixed points
for $\sigma$. Thus the ramification
is the union of the curves
$B_{p_1}\dots B_{p_{2g+2}}$
We conclude that the branch divisor of $\psi$
is the union of the $2g+2$ lines
$L_{p_1},\dots,L_{p_{2g+2}}$
which, by Remark~\ref{rem:tan},
are tangent to $\Gamma$ at the 
images of  $p_j\oplus p_j$.

\end{proof}

\begin{remark}
\label{q}
Since $X$ is a hypersurface of a weighted projective 
space, by~\cite{Do}*{Thm. 4.2.2} we have that $q(X) = 0$. 
In particular $X$ is a weak del Pezzo of
degree $2$ if $g=1$, it is a singular $K3$
when $g=2$ and it is of general type when
$g\geq 3$.
\end{remark}

\begin{proposition}
 \label{prop:C2X}
Assume that $C$ is a very general 
hyperelliptic curve of genus $g\geq 1$. 
Then both $C^{(2)}$ and $X$ have Picard rank 
$2$ and their effective cones are generated by
the classes of the images of the diagonal and
the antidiagonal. 
The intersection matrices of these curves 
in $C^{(2)}$ and of their images in $X$ are
\[
 \begin{pmatrix}
  4-4g&2g+2\\
  2g+2& 1-g
 \end{pmatrix}
 \qquad
 {and}
 \qquad
 \begin{pmatrix}
  2-2g&2g+2\\
  2g+2&2-2g
 \end{pmatrix}
\]
respectively.
\end{proposition}
\begin{proof}
By~\cite{acgh}*{Chapter~VIII, $\S 5$},
 $\Num^1(C^{(2)}) 
 \simeq
 \mathbb Z\oplus \Num^1(JC)$,
so that ~\cite{pi}*{Prop. 3.4}
implies that the Picard rank of 
$C^{(2)}$ is $2$.
As a consequence the Picard rank of $X$ 
is at
most $2$ and it is $2$ because $N^1(X)$ 
contains two numerically independent classes.
The diagonal $\Delta_+$ and the antidiagonal $\Delta_-$
are both mapped to the conic $\Gamma$ of $\mathbb P^2$,
tangent to the $2g+2$ lines. Thus $C_++C_- = \phi(\Delta_+)
+\phi(\Delta_-)$ is the 
pullback of $\Gamma$, so that $(C_++C_-)^2 = 8$. 
Since these two curves are numerically equivalent and
intersect in $2g+2$ points, we obtain the
second matrix. To get the first matrix it is
enough to observe that the double cover 
$C^{(2)}\to X$ branches at $C_-$, which is
the image of the antidiagonal.
\end{proof}

\section{Proof of Theorem~\ref{thm:2}
and~\ref{thm:1}}
\label{sec:proof}
\begin{proof}[Proof of Theorem~\ref{thm:2}]
Let us fix a point $p\in C$, such that 
the class of $\sigma(p) - p$
is non-torsion, and 
let $\tilde C^{(2)}\to C^{(2)}$ be 
the blowing-up at the point 
$p\oplus\sigma(p)\in \Delta_-\cap C_{\sigma(p)}$, 
with exceptional divisor $E$.
First of all observe that the point $p\oplus\sigma(p)$
is invariant for $\sigma$, so that the latter
lifts to an involution on the blow-up 
$\tilde C^{(2)}$ that, by abuse of notation, we 
denote by the same symbol $\sigma$.
Let $\tilde\phi\colon \tilde C^{(2)}\to \tilde X:= \tilde C^{(2)}/\langle\sigma\rangle$ 
be the quotient morphism.
The involution $\sigma$ has two fixed points 
on the exceptional divisor $E$: the intersection point 
with the strict transform of $\Delta_-$, and one isolated point $x$, 
so that $\tilde\phi(x)$ is a singular point of $\tilde X$.
We have a birational morphism 
$\eta\colon\tilde X\to X$ which is the contraction of $\tilde\phi(E)$, having self-intersection 
$-1/2$ in $\tilde X$.
The map $\eta$ is a weighted blow-up at the point $\phi(\sigma(p)\oplus p)$ and  can also be described as follows.
Consider the blow-up $X_1\to X$ at the point $\phi(p\oplus\sigma(p))$, 
with exceptional divisor $E_1$, and then the blow-up $X_2\to X_1$, 
at the intersection point of $E_1$ with the strict transform of 
$B_p = \phi(C_p) = \phi(C_{\sigma(p)})$ (see Remark~\ref{rem:tan}). 
Finally contract the strict transform of $E_1$, which is now a $(-2)$-curve
(its image gives the singular point $\tilde\phi(x)\in \tilde X$).
We can resume the above discussion in the following commutative diagrams:
\[
 \xymatrix{
  {\tilde C}^{(2)}\ar[r]^-{\tilde\phi}\ar[d]
  & {\tilde X}\ar[d]^-{\eta}
  & &  X_2\ar[d]\ar[r]
  & \tilde X\ar[d]^-{\eta}
  & \\
  C^{(2)}\ar[r]_-{\phi}
  & X\ar[r]_-{\psi}
  & \pp^2
  &
    X_1 \ar[r]  & X.
 }
\]
We are going to show that 
the pseudo-effective cone of $\tilde{X}$ is not polyhedral. 
Observe that the strict transform
of $B_p$ in $\tilde X$ is isomorphic 
to $B_p$ and hence to $C_{\sigma(p)}$ and to $C$. Therefore, by abuse of notation, we denote this strict transform by $C$.
Since $\psi(B_p)\subseteq \pp^2$ is the
line $L_p$, tangent to the conic $\Gamma$, we have that $B_p^2 = 2$. 
By the description of $\eta\colon \tilde X\to X$,
we are blowing up a point on $B_p$ and then the same point on its strict transform.
Therefore $C^2 = 0$, and we can write
\begin{equation}
    \label{eq:C}
 C\sim (\psi\circ\eta)^*(L_p) - 2\tilde\phi(E).
\end{equation}
Let us compute now the restriction $\mathcal O_C(C)$. 
Since $L_p$ is a line in $\pp^2$, 
the restriction of 
$(\psi\circ\eta)^*(L_p)$ to $C$ is the $g_2^1$, so that it is equivalent to
$\sigma(p) + p$.
On the other hand, the restriction of $\tilde\phi(E)$ to $C$ corresponds
to the point we are blowing-up in $\eta\colon\tilde X \to X$, i.e. 
to the image of $\sigma(p)\oplus p$ in $X$.
Via the isomorphism $C_{\sigma(p)}\to C$, the point 
$\sigma(p)\oplus p$ corresponds to $p\in C$, so
that we conclude that the restriction of 
$\tilde\phi(E)$ to $C$ is $p$.
Summing up we obtain
\[
 \mathcal O_C(C)
 \simeq
 \mathcal O_{C}(\sigma(p)+p-2p)
 =
 \mathcal O_{C}(\sigma(p)-p).
\]
Since we are assuming that $\sigma(p)-p$ is non-torsion, 
we deduce that $ \mathcal O_C(C)$ is non-torsion, i.e.
the class of $D$ spans an extremal ray of
$\barEff(\tilde X)$, so that this cone 
is not polyhedral.
By Proposition~\ref{prop:eff1} we conclude that 
the cone $\barEff(\tilde C^{(2)})$ is not polyhedral.
\end{proof}

\begin{remark}
\label{rem:jenia}
In genus 2 the Abel-Jacobi map presents 
$C^{(2)}$ as the blow-up of ${\rm Pic}^2C$ 
in the point $\Omega$ that corresponds to the canonical
class $K_C$ of $C$, with exceptional divisor
$\Delta_-\subseteq C^{(2)}$.
So in this case we blow-up ${\rm Pic}^2C$ twice
infinitely near at $\Omega$. 
The map $C\to {\rm Pic}^2C$ given by 
$x\mapsto [x+\sigma(p)]$ embeds $C$ as a theta-divisor
passing through $\Omega$ and with tangent direction 
$p+\sigma(p)$. So after the blow-up the proper transform
of $C$ is a curve of self-intersection $0$. The
restriction of $C$ to $C$ will be $K_C-2p$ (because we
blow-up twice the same point $p$ of $C$), so it is $\sigma(p)-p$ as we claim in the proof of the general 
statement.
\end{remark}

\begin{remark}
Assume that $\sigma(p) - p$ is not
torsion, so that the pseudo-effective cone
$\barEff(\tilde X)$ is not
polyhedral. 
If $g = 1$ we are going to 
show that on $\tilde X$ there 
are infinitely
many negative rays accumulating
on $C$ (see
Proposition~\ref{prop:dn}
and Figure~\ref{fig:eff}).
If $g > 1$, consider the
intersection matrix of the 
classes $\Delta_+,\Delta_-,E_1,
C,E$ on $\tilde X$:
\[
 \begin{pmatrix}
  2-2g & 2+2g & 1 & 2 & 0\\
  2+2g & -2g & 1 & 0 & 1\\
  1 & 1 & \frac12 & 1 & 0\\
  2 & 0 & 1 & 0 & 1\\
  0 & 1 & 0 & 1 & -\frac12
 \end{pmatrix}.
\]
We already know that $C$
generates an extremal ray
of $\barEff(\tilde X)$, and the
same holds for the classes 
$\Delta_+, \Delta_-$ and $E$, 
since they have negative 
self-intersection. In particular 
$\barEff(\tilde X)$ (and hence also
$\barEff(\tilde C^{(2)})$) has a 
polyhedral part
(see Figure~\ref{fig:g>1}).
\end{remark}

\begin{figure}[htbp]
\hspace{5mm}
\includegraphics[width=7cm]{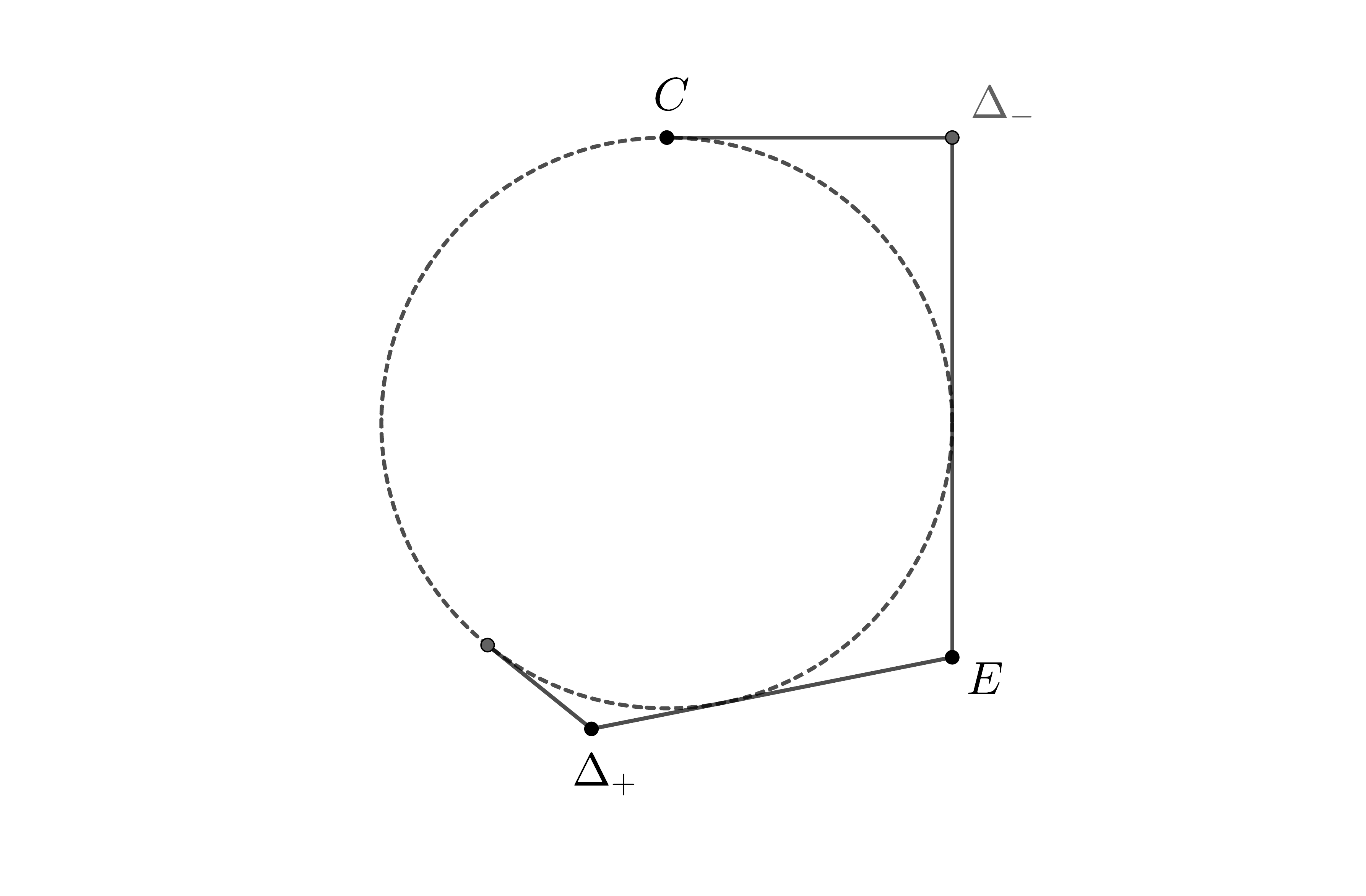}
\caption{$\Eff(\tilde X)$, when 
$g > 1$} 
\label{fig:g>1}
\end{figure}

\begin{question}
When $C$ is hyperelliptic of genus $g > 1$, 
does $\tilde C^{(2)}$ have infinitely many 
negative curves?
\end{question}

\begin{proof}[Proof of Theorem~\ref{thm:1}]
Let $\pi\colon\mathcal X\to B$ be a flat
family whose general fiber is a general
genus $g$ curve $C$ and whose special fiber
over $0\in B$ is a general hyperelliptic 
curve $C_0$. Passing to the symmetric product
one gets a new flat family with basis $B$.
Blowing-up a section of the new family
which cuts out $p \oplus \sigma(p)$ on
$C_0^{(2)}$, with $\sigma(p)-p$ non-torsion, one concludes, by Theorem~\ref{thm:2} and
Proposition~\ref{prop:eff3},
that the 
pseudo-effective cone of the blow-up 
$\tilde C^{(2)}$ is non-polyhedral.
\end{proof}

We remark that when $C$ and the point that we are 
blowing up are general, even if we
know that $\barEff(\tilde C^{(2)})$ is not 
polyhedral, we do not know any negative class.
Therefore it is natural to ask the 
following.
\begin{question}
 When $C$ is general, does $\tilde C^{(2)}$ have infinitely many negative curves?
\end{question}

\section{The genus one case}
\label{sec:ell}
In this section we make the assumption
that $C$ has genus $1$. 
In particular we first show that in Theorem~\ref{thm:1}
also the opposite implication holds (see Theorem~\ref{thm:3}).
Then we describe
the rays of the pseudo-effective cone 
of $\tilde X$, both when it is polyhedral and when it
is not (Proposition~\ref{prop:dn}), and finally we give 
a planar model for the resolution $Z$ of $\tilde X$.

\begin{remark}
 \label{rem:int}
When $g(C) = 1$, the symmetric product $C^{(2)}$ is a ruled surface
whose fibers correspond to the $g_1^2$'s of $C$.
Observe that if we fix two points $p\neq q \in C$,
they define a unique $g_1^2$, and hence a
hyperelliptic involution $\sigma$. This implies
that the antidiagonal $\Delta_- = \{r\oplus \sigma(r) 
\mid r\in C\}$ is indeed a fiber.
\end{remark}

Let us recall the following definition from~\cite{CLTU}*{$\S~3$}.
\begin{definition}\label{def:pair}
An  {\em elliptic pair} $(C,X)$ consists of  a projective
rational surface $X$ with log terminal singularities and an irreducible curve $C\subseteq X$, with arithmetic genus one,  
disjoint from the singular locus of $X$ and such that $C^2=0$.

The elliptic pair $(C,X)$ is a {\em minimal elliptic pair} if
it does not contain irreducible curves $E$ such that $K\cdot E < 0$ 
and $C\cdot E = 0$.
\end{definition}

Let us consider as before the blowing-up 
$\tilde C^{(2)}\to C^{(2)}$ 
at the point $p\oplus q\in\Delta_-$, with $p\neq q$, or equivalently at 
$p\oplus \sigma(p)$, where $\sigma$ is the involution exchanging $p$ and $q$.
We denote by $E$ the exceptional divisor and by $\tilde C_p\subseteq \tilde C^{(2)}$
the strict transform of the curve $C_p:=\{p\oplus r \mid r\in C\}\subseteq C^{(2)}$.
The involution $\sigma$ induces an involution on $\tilde C^{(2)}$, that we still denote 
by $\sigma$, whose ramification is the strict transform
$\tilde\Delta_-\subseteq \tilde C^{(2)}$ of $\Delta_-$. 
We denote by $\tilde\phi\colon \tilde C^{(2)}\to \tilde X := \tilde C^{(2)}/\langle\sigma\rangle$ 
the quotient morphism. 
Since the curve $\tilde\phi(\tilde C_p)$ is isomorphic 
to $C_p$ and hence to the curve $C$, by abuse of notation 
in what follows we will simply set $C := \tilde\phi(\tilde C_p)$.


\begin{lemma}
\label{lem:ell-pair}
The pair $(C,\tilde X)$ is a minimal elliptic pair.
\end{lemma}
\begin{proof}
The rationality of $\tilde X$ follows from Remark~\ref{q}. 
From Proposition~\ref{prop:S1} we have that $X$ has $6$ ordinary 
double points, and none of them lies on 
$B_p:=\phi(C_p) = \phi(C_q)$.
Therefore they give rise to $6$ ordinary double points
of $\tilde X$, disjoint from the curve $C$.
Moreover the involution on $\tilde C^{(2)}$ has $2$ fixed points
on $E$, but only one of them is isolated. Its image is the seventh 
ordinary double point of $\tilde X$ (which does not lie on $C$).
This proves that  $(C,\tilde X)$ is an elliptic pair. By 
Proposition~\ref{prop:tildeX}
we can compute $K_{\tilde X}^2 = 0$, so that
by~\cite{CLTU}*{Lemma~3.7} we conclude that $(C,\tilde X)$ is minimal.

\end{proof}

\begin{remark}
\label{rem:geom}
Let us consider a minimal resolution
$\pi \colon Z\to \tilde X$. Since $C\subseteq 
\tilde X$ does not pass through the
singular points, we have an isomorphic 
copy of $C$ in $Z$, that we still denote by $C$. Therefore $(C,Z)$
is a smooth minimal elliptic pair and
in particular, by~\cite{CLTU}*{Theorem~3.8},
the Picard rank of $Z$ is $10$.
\end{remark}

\begin{notation}
\label{not:XZ}
Before stating our next results about $\tilde X$ and $Z$ we need to fix some notation.
First of all we are going to denote by 
$L_i\subseteq \pp^2$, $1\leq i\leq 4$ the
lines whose union is the branch locus of $X\to\pp^2$, and by $E_i \subseteq \tilde X$
and $\bar{E_i} \subseteq Z$ the strict transforms of $L_i$ on $\tilde X$ and $Z$ respectively.
By abuse of notation we denote by $E$
the image $\tilde \varphi(E)\subseteq \tilde X$ 
and by $\bar E \subseteq Z$
its strict transform. Analogously
we denote simply $\Delta_-$ the curve  $\tilde\phi(\tilde\Delta_-)\subseteq \tilde X$ 
and by $\bar \Delta_-$ its strict transform in
$Z$.
For any $1\leq i < j\leq 4$, we denote by
$\bar E_{ij}\subseteq Z$ the $(-2)$-curve
over the singular point $p_{ij} := L_i\cap L_j \in \pp^2$, while the $(-2)$-curve over the 
isolated singular point $\tilde\phi(x) \in \tilde\phi(E)\in\tilde X$ is denoted
by $\bar E'$.
Finally, for any $(i,j,k)$ in $\{(1,2,3),(1,3,2),
(2,3,1)\}$,
we denote by $L_{(ij)(k4)}\subseteq 
\pp^2$ the line through  $p_{ij}$ and $p_{k4}$. 
Observe that over any of these $L_{(ij)(k4)}$ we
have two irreducible curves in $X$, say
$E_{(ij)(k4)}$ and $E'_{(ij)(k4)}$. We use
the same notation for their strict transforms
on $\tilde X$, while we denote by 
$\bar E_{(ij)(k4)}$ and $\bar E'_{(ij)(k4)}$
their strict transforms on $Z$.

We now recall that $\Gamma =
V(x_3^2-4x_1x_2) \subseteq \pp^2$, so that 
$\Delta_- = V(x_4 +
2x_1x_3 - 2x_2x_3)$
(it corresponds to one of
the two irreducible 
components over $\Gamma$).
Moreover we can 
fix the tangent lines
$L_1,\dots, L_4$ to be
$V(x_1),\, V(x_2), \, V(x_1+x_2-x_3)$
and $V(x_1+x_2+x_3)$ respectively.
Then, by Proposition~\ref{prop:S1}, 
\[
X = V(x_4^2 - x_1x_2(x_1+x_2-x_3)(x_1+x_2+x_3)) \subseteq\pp(1,1,1,2),
\]
and 
$L_{(12)(34)} = V(x_1+x_2),\,
L_{(13)(24)} = V(x_1-x_2+x_3),\,
L_{(23)(14)} = V(x_1-x_2-x_3)$. 
From these equations one can see
that the $6$ curves $E_{(ij)(kl)}$ and
$E'_{(ij)(kl)}$ form a hexagon, and 
we can choose the labels in order to have
\begin{align*}
    E_{(12)(34)} & = 
    V(x_1 + x_2, x_4 + x_2x_3),\\
    E_{(13)(24)} & = 
    V(x_1 - x_2 + x_3, 
    x_4 + 2x_2^2 - 2x_2x_3),\\
    E_{(23)(14)} & = 
    V(x_1 - x_2 - x_3, 
    x_4 - 2x_2^2 -
    2x_2x_3).
\end{align*}
It is now straightforward 
to check that these $3$
curves are disjoint and 
do not meet $\Delta_-$, 
so that the same holds for
the strict transforms on 
$\tilde X$ and on $Z$
(analogously, $E_{(12)(34)}',\ E_{(13)(24)}'$
and $E_{(23)(14)}'$ are disjoint and do not
meet $\Delta_+$).

In Picture~\ref{fig:intZ} we represent the intersection 
products of the negative curves described before. The red dots
are the $(-2)$-curves while the blue dots are the $(-1)$-curves.
When two dots are connected, the two corresponding curves
have intersection product $1$, otherwise their product is $0$.

\end{notation}

\begin{remark}
The lattice $C^{\perp}$ in
$\Pic(Z)$ is isomorphic to 
$\tilde{\mathbb E}_8$. Since the eight $(-2)$-curves 
described above are all disjoint, their classes span the sublattice 
$\mathbb A_1^8\subseteq \tilde{\mathbb E}_8$.
\end{remark}

\begin{figure}
\begin{center}
\begin{tikzpicture}[scale=1.5,x={(5mm,0mm)},z={(0mm,5mm)},y={(0mm,2mm)},
main node/.style={circle,fill=red!70,draw,minimum size=5pt,inner sep=1pt},
0 node/.style={circle,draw,minimum size=5pt,inner sep=1pt},
1 node/.style={circle,fill=blue!70,draw,minimum size=5pt,inner sep=1pt}]

\node[main node] (1) at (2,0,0) {} 
node[below left] at (2,0,0) {\tiny $\bar E_{12}$};
\node[main node] (2) at (6,0,0) {}
node[below right] at (6,0,0) {\tiny $\bar E_{13}$};
\node[main node] (3) at (8,4,0) {}
node[right] at (8,4,0) {\tiny $\bar E_{24}$};
\node[main node] (4) at (6,8,0) {}
node[above right] at (6,8,0) {\tiny $\bar E_{23}$};
\node[main node] (5) at (2,8,0) {}
node[above left ] at (2,8,0) {\tiny $\bar E_{14}$};
\node[main node] (6) at (0,4,0) {}
node[left] at (0,4,0) {\tiny $\bar E_{34}$};
\node[main node] (14) at (5,4,6) {}
node[above left] at (5,4,6) {\tiny $\bar \Delta_-$};
\node[main node] (15) at (9,4,6) {}
node[above right] at (9,4,6) {\tiny $\bar E'$};

\path[every node/.style={font=\sffamily\small}, dashed]
(1) edge (2)
(2) edge (3)
(3) edge (4)
(4) edge (5)
(5) edge (6)
(6) edge (1)
(14) edge (15);

\node[1 node] (7) at (4,0,0) {}
node[below] at (4,0,-0.1) {\tiny $\bar E_1$};
\node[1 node] (8) at (7,2,0) {}
node[below right] at (7,2,0) {\tiny $\bar E_{(13)(24)}$};
\node[1 node] (9) at (7,6,0) {}
node[above right] at (7,6,0) {\tiny $\bar E_2$};
\node[1 node] (10) at (4,8,0) {}
node[above] at (4,8,0) {\tiny $\bar E_{(23)(14)}$};
\node[1 node] (11) at (1,6,0) {}
node[above left] at (1,6,0) {\tiny $\bar E_4$};
\node[1 node] (12) at (1,2,0) {}
node[below left] at (1,2,0) {\tiny $\bar E_{(12)(34)}$};
\node[1 node] (13) at (4,4,0) {}
node[below left] at (4.2,4,0) {\tiny $\bar E_3$};
\node[1 node] (16) at (7,4,6) {}
node[above] at (7,4,6) {\tiny $\bar E$};

\path[every node/.style={font=\sffamily\small}, dashed]

(13) edge (2)
(13) edge (4)
(13) edge (6)

(1) edge (9)
(3) edge (11)
(5) edge (7)

(7) edge (14)
(9) edge (14)
(11) edge (14)
(13) edge (14);




\end{tikzpicture}
\end{center}
\caption{Intersection graph on $Z$} 

\label{fig:intZ}
\end{figure}
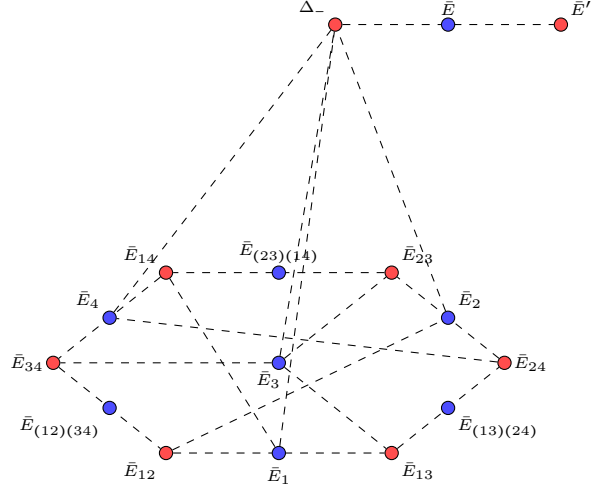

\begin{proposition}
 \label{prop:tildeX}
On $\tilde X$ the following hold.
\begin{enumerate}
\item
${\rm Cl}(\tilde X)\simeq\mathbb Z^3\oplus(\mathbb Z/2\mathbb Z)^2$.
\item
${\rm Cl}(\tilde X)_{\rm free}$ is generated by $E_1$, $E_{(12)(34)}$
and $E$, which have the following 
intersection matrix
\[
 \begin{pmatrix}
  1/2 & 1/2 & 0\\
  1/2 & 0 & 0\\
   0  & 0 & -1/2
 \end{pmatrix}.
\]
\item
$C\sim -K_{\tilde X}\sim 2E_1 - 2E$.
\item
$\Delta_-\sim 2E_{(12)(34)}-2E$. 
\end{enumerate}
\end{proposition}
\begin{proof}
We prove $(1)$.
First of all observe that the Picard rank of $\tilde X$ is
$3$ because we are contracting seven $(-2)$-curves
of $Z$, which has Picard 
rank $10$ (see Remark~\ref{rem:geom}).
Moreover the torsion part is of the form 
$(\mathbb Z/2\mathbb Z)^s$, for some $0\leq s\leq 7$,
because the singularities of $\tilde X$ are ordinary
double points.
A basis for the Picard group of $Z$ consists
of the classes of the following curves:
$\bar E_{12}$, $\bar E_{13}$, $\bar E_{14}$, 
$\bar E_1$, $\bar E_2$, $\bar E_3$, $\bar E_4$,
$\bar E_{(12)(34)}$, $\bar E'$, $\bar E$ 
because the corresponding 
intersection matrix is unimodular of rank $10$.
This implies that ${\rm Cl}(\tilde X)$ is generated by 
their images, i.e. $E_1$, $E_2$, $E_3$, $E_4$, $E_{(12)(34)}$, $E$ 
(recall Notation~\ref{not:XZ}).
Observe that for any $i\neq j$, the class of $E_i-E_j$ is $2$-torsion
because $2E_i$ is pullback of a line of $\mathbb P^2$.
Since the class of $E_1+E_2+E_3+E_4$ is
linearly equivalent to the pullback of a 
conic, it is divisible by $2$, so that also the class
of 
\[
 (E_2-E_1)+(E_2-E_3)+(E_2-E_4) 
 = 4E_2-(E_1+E_2+E_3+E_4)
\]
is divisible by 
$2$, and in particular it
is trivial. Therefore 
the class of $E_4$ is not needed to generate
${\rm Cl}(\tilde X)$ and thus $s\leq 2$.
On the other hand $E_1-E_2\neq E_1-E_3$
so that $s=2$.

We prove $(2)$. Since $E$ is disjoint from
$E_1$ and $E_{(12)(34)}$, we have 
$E_1\cdot E = E_{(12)(34)}\cdot E = 0$.
The self-intersection of $E_1$
is $1/2$ because $2E_1$ is the pullback
of a line. The self-intersection of 
$E_{(12)(34)}$ is $0$ because its pullback
in $Z$ is $1/2E_{12}+E_{(12)(34)}+1/2E_{34}$,
which has self-intersection $0$.
Similarly one shows that $E^2=-1/2$ and that
$E_{(12)(34)}\cdot E_1 = 1/2$.

We prove $(3)$. The
equivalence $C\sim 2E_1 - 2E$ 
follows from equation~\eqref{eq:C}, 
since $L_p\subseteq \pp^2$ is a line
tangent to $\Gamma$. 
By the ramification formula,
\[
 K_X = \psi^*K_{\mathbb P^2}+R 
 = 
 \psi^*(-3L)+\psi^*(2L) = -\psi^*(L)
 = 
 -2\eta(E_1).
\]
Recall that the map 
$\eta\colon\tilde X\to X$ is obtained by 
blowing up twice (one time on the exceptional
divisor) and then contracting the $(-2)$-curve.
The contraction is crepant so that it does
not affect the canonical class. From this we
conclude that
\[
 K_{\tilde X} = \eta^*K_X + 2E
 = -2E_1 + 2E.
\]
In order to prove (4) observe that on $Z$
the divisors $2\bar E_{(12)(34)}+
\bar E_{34}+\bar E_{12}$ 
and $\bar\Delta_-+2\bar E+\bar E'$
have both self-intersection $0$ and 
their intersection product is $0$.
By the Hodge index theorem it follows that
the classes of these divisors must be
proportional. Since both classes have 
intersection product $1$ with some curve,
they are primitive
in ${\rm Pic}(Z)$. 
It follows that the two classes are equal, and 
one concludes by taking pushforward
of these classes via $\pi\colon Z\to\tilde X$.

\end{proof}

\begin{theorem}
\label{thm:3}
With the notation above, the following are equivalent:
\begin{enumerate}
\item
${\rm Eff}(\tilde C^{(2)})$ is rational polyhedral;
\item
${\rm Eff}(\tilde X)$ is rational polyhedral;
\item
${\rm Eff}(Z)$ is rational polyhedral;
\item
the class of $q - p$ has order $m <\infty$ 
in ${\rm Pic}^0(C)$;
\item
$\dim |-mK_Z| = 1$ and $\dim |-rK_Z| = 0$, for
$0\leq r < m$.
\end{enumerate}
\end{theorem}

\begin{proof}
By Proposition~\ref{prop:C2X}, $\rho(\tilde X) =
\rho(\tilde C^2)$, so that the equivalence (1) $\Leftrightarrow$
(2) follows from
Proposition~\ref{prop:eff1}. 
Since $\tilde X$ has Du Val singularities,
the equivalence (2) $\Leftrightarrow$ (3) 
was proved in~\cite{CLTU}*{Lemma~3.14}.
We now prove the equivalence of (3) and (4).
Since $(C,Z)$ is an elliptic
pair, by~\cite{CLTU}*{$\S 3$}, the effective cone of $Z$ 
is rational polyhedral if and only if
$C^\perp$ is generated by the kernel of 
$\rd \colon C^\perp\to {\rm Pic}^0(C)$.
In Remark~\ref{rem:geom} we have already
seen that there are eight disjoint 
$(-2)$-curves in $\ker(\rd)$.
Thus $C^\perp$ is spanned by elements 
of $\ker(\rd)$ if and only if there
exists an integer $m > 1$ such that
the multiple $mC$ is in $\ker(\rd)$,
that is if $\rd(C)$ is of $m$-torsion.
We conclude by observing that   
${\rm res}(C) = q - p$ (see the proof
of Theorem~\ref{thm:2}).

Finally, from Proposition~\ref{prop:tildeX}
$C \sim -K_{\tilde X}$, and since 
the curve $C$ is disjoint from
the singular points, also on $Z$ we have
$C\sim -K_Z$. The equivalence (4)
$\Leftrightarrow$ (5) follows.
\end{proof}

We are now going to describe
the extremal rays of 
$\Eff(\tilde X)$, both when
it is polyhedral and when
it is not. We remark that
by Proposition~\ref{prop:eff1}
we can identify $\Eff(\tilde C^{(2)})$
with $\Eff(\tilde X)$.

\begin{proposition}
\label{prop:dn}
Let $m > 1$ be the order of $q - p$
and let us consider the following classes:
\begin{align*}
 D_n &:= 2n(n+1)E_1-2nE_{(12)(34)}+(1-2n^2)E,
 \quad n\in \mathbb Z_{\geq 0}.
 \end{align*}
\begin{enumerate}
 \item If $m = \infty$, then
   ${\rm Eff}(\tilde X) =
  \langle -K_{\tilde X},
  \Delta_-,D_0,D_1,\dots,D_n,\dots\rangle$.
 \item If $m < \infty$, then
  ${\rm Eff}(\tilde X) = \langle
  -K_{\tilde X},
  \Delta_-,D_0,D_1,\dots,D_{\lceil
  \frac{m}2\rceil -1},\Gamma_m\rangle$,
  where 
 \[
 \Gamma_m  :=
 mE_1 - E_{(12)(34)} + (1-m)E.
 \]

\end{enumerate}

\end{proposition}
\begin{proof}
(1) If $m = \infty$, we already know
that the effective cone of 
$\tilde X$ is not polyhedral.
A direct calculation shows that $D_n^2 = -1/2$
and $D_n\cdot K_{\tilde X} = -1$,
for any $n\geq 0$.
The divisors $2E_1$ and $2E_{(12)(34)}$ are
Cartier, while $E$ is not Cartier. 
Since $\bar E'$ is the only $(-2)$-curve intersecting $\bar E$ and contracted 
by $\pi\colon Z\to \tilde X$, it follows that
\[
 R_n 
 := 
 \lfloor\pi^*D_n\rfloor
 =
 \pi^*D_n-\frac{1}{2}\bar E'
\]
is a divisor with integer coefficients.
Since $R_n^2 = R_n\cdot K_Z = -1$, 
by Riemann-Roch we conclude 
that $R_n$ is linearly
equivalent to an effective divisor.
Moreover each $R_n$ has non-negative intersection 
product with all the $(-2)$-curves since 
$R_n\cdot \bar E_{ij} = 0$, $R_n\cdot \bar E' = 1$ and
$R_n\cdot \bar\Delta_- = 2n + 1$.
We claim that $R_n$ is irreducible. 
Suppose that we can write $R_n = C_1 + N$,
where $C_1$ is an irreducible $(-1)$-curve
and $N$ is a sum of $(-2)$-curves. 
The condition $R_n^2 = -1$ implies that
either $R_n\cdot C_1 < 0$ or 
$R_n\cdot N <0$, but the latter would
imply that $R_n$ has negative intersection
with at least one $(-2)$-curve, 
a contradiction. Therefore $(C_1+N)\cdot
C_1 < 0$, so that $N\cdot C_1 = 0$, which
implies that also $N^2 = 0$. 
Since the intersection form is negative
semidefinite on the components of $N$, 
we deduce that $N$ is indeed a multiple
of $-K$. Therefore $R_n = C_1 - tK$, which
gives $R_n^2 > -1$, again a contradiction.
This proves the claim, and since
$D_n = \pi_*(R_n)$, it is
irreducible too.

Consider now the cone $\mathcal C$,
generated by 
$-K_{\tilde X},\, \Delta_-$ and $D_n$, for $n\geq 0$.
The following matrices
\[
\left(
\begin{array}{rr}
0 & 0\\
0 & -2
\end{array}
\right),
\qquad
\left(
\begin{array}{rr}
-2 & 1\\
1 & -\frac12
\end{array}
\right),
\qquad
\left(
\begin{array}{rr}
-\frac12 & \frac12\\
\frac12 & -\frac12
\end{array}
\right)
\]
give the intersection form 
on the edges $\langle -K_{\tilde X},\Delta_-\rangle,\,
\langle \Delta_-,D_0\rangle$ and 
$\langle D_n,D_{n+1}\rangle$ 
(for any $n\geq 0$) respectively.
Since they are all negative semidefinite
and the rays $D_n$ accumulate on $-K_{\tilde X}$,
we conclude that 
$\mathcal C = \Eff(\tilde X)$,
which proves (1).

Let us prove (2). Observe that if 
$m < \infty$, then $|-mK_Z|$ 
defines an elliptic fibration which is {\em 
extremal} in the sense of Miranda-Persson~\cite{mp}.
According to~\cite{mp}*{Thm. 4.1} the only 
extremal rational elliptic surface which contains
eight disjoint $(-2)$-curves is $X_{11}(j)$, which
has exactly two singular fibers of type $I_0^*$.
Thus, as soon as $-mK_Z$ moves, two new $(-2)$-curves
appear, each of which is the unique curve of 
multiplicity two in the fiber $I_0^*$.
On $\tilde X$ one of these curves is disjoint from
$\Delta_-$, so that its self-intersection is $0$, 
while the other one intersects $\Delta_-$ and
has self-intersection $-1/2$.
The class of the latter is
\[
 \Gamma_m
 :=
 \frac{1}{2}(-mK_{\tilde X}-\Delta_-)
 =
 mE_1 - E_{(12)(34)} + (1-m)E,
\]
and by the intersection matrix given in
Proposition~\ref{prop:tildeX} we have that 
$\Gamma_m\cdot D_n = 1/2(m-1) - n$, which is non-negative if and
only if $n < {\lceil\frac{m}2\rceil}$.

\end{proof}

\begin{figure}[htbp]
\hspace{-5mm}
\includegraphics[width=7cm]{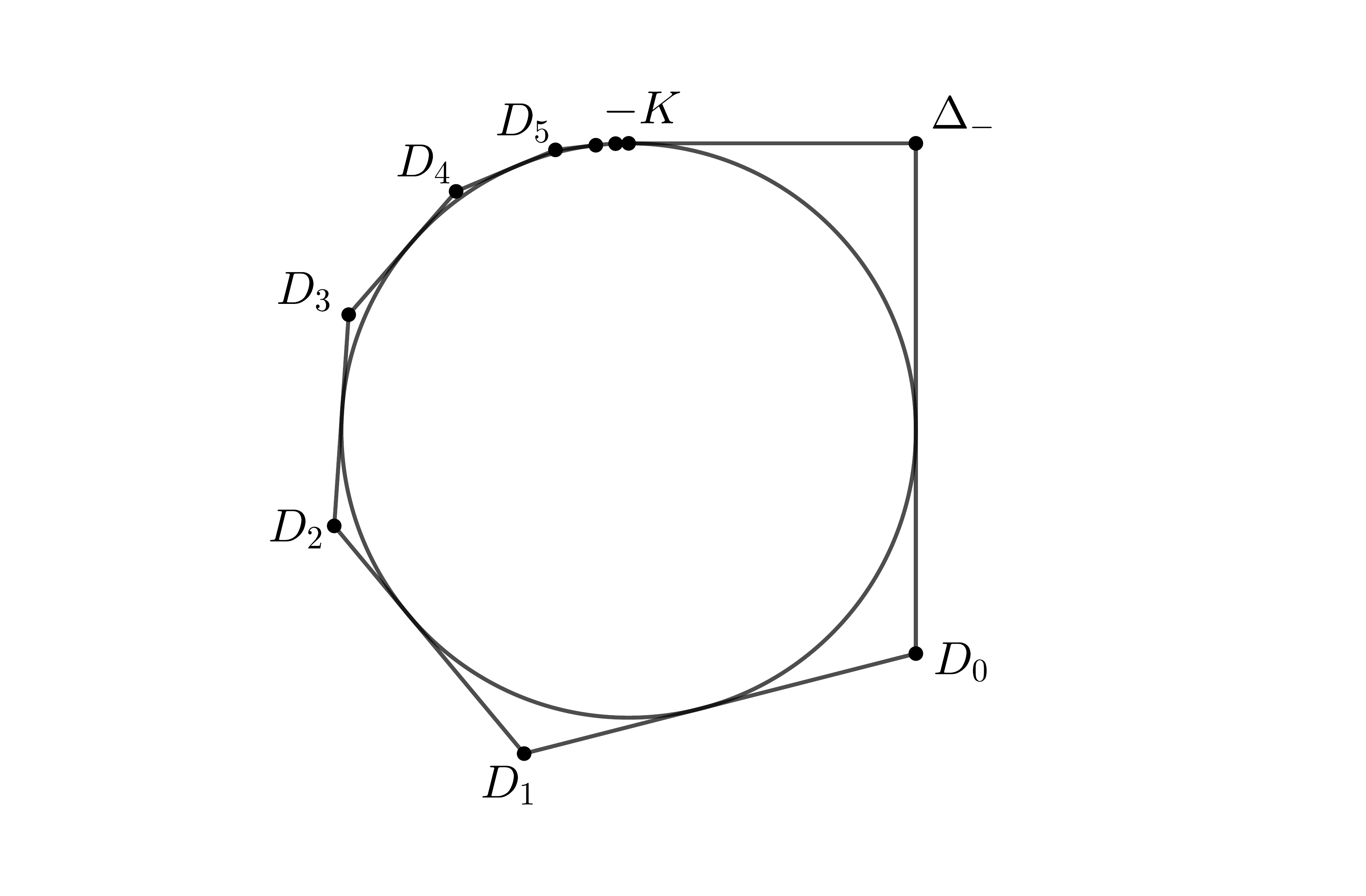}
\hspace{-2cm}
{\includegraphics[width=6.8cm]{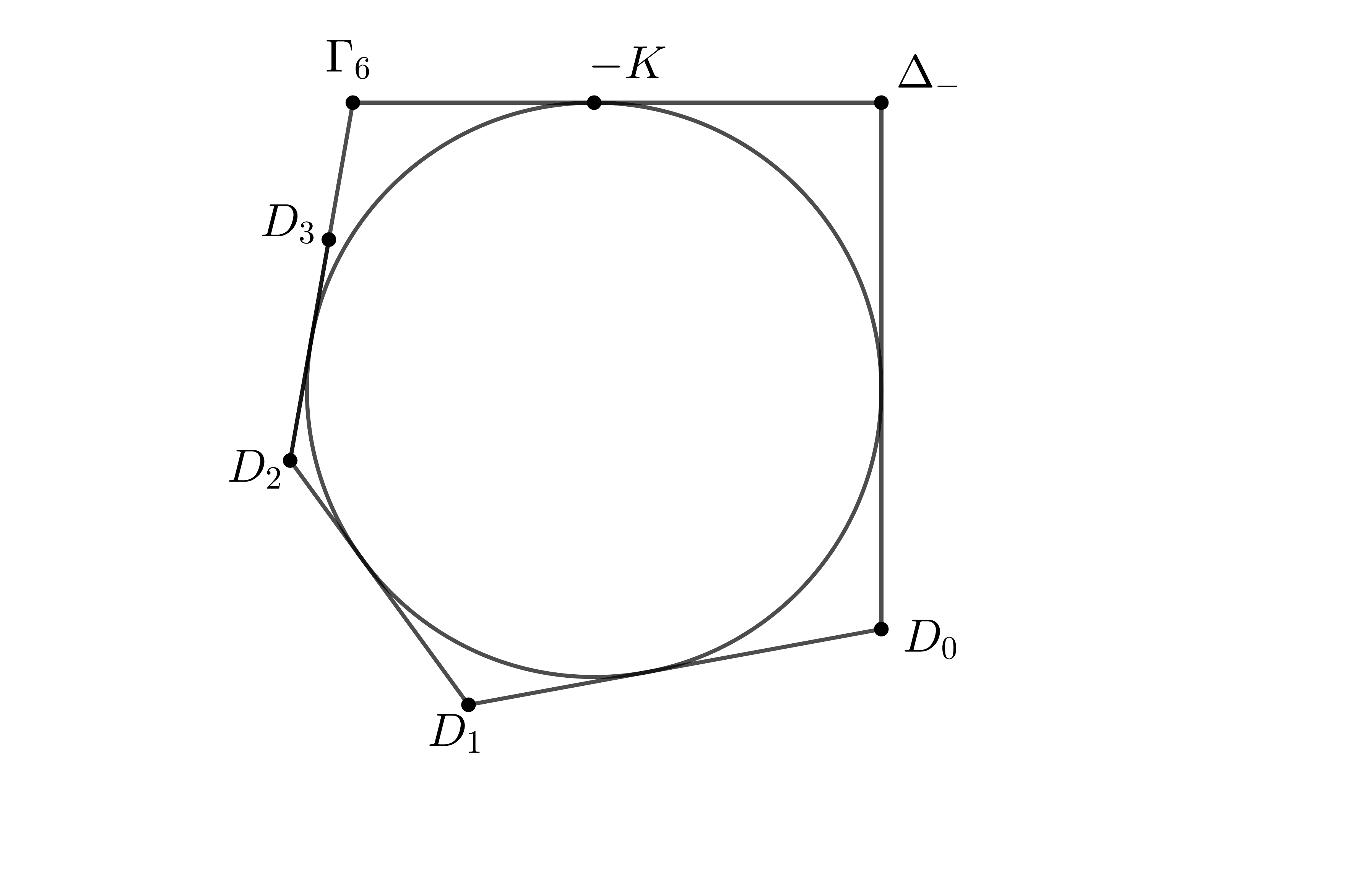}}
\hspace{-3cm}
\caption{$\Eff{\tilde X}$ for $m=\infty$ and $m = 6$} \label{fig:eff}
\end{figure}

\begin{remark}
We remark that
the irreducible negative curves on 
$\tilde C^{(2)}$, first found in~\cite{Ga},
are {\em unexpected}, meaning with this 
that the expected dimension of the linear
system is negative.
The images of these curves in $\tilde X$
are the curves $D_n$ of
Proposition~\ref{prop:dn}. 
They are still negative,
of self-intersection $-\frac{1}{2}$, 
but we have seen that the round-down 
of the pull-back $R_n = \lfloor\pi^*D_n\rfloor$ 
on $Z$ is a $(-1)$-curve, and in particular
it is expected.
\end{remark}

\begin{remark}
 \label{rem:plane}
Looking at Figure~\ref{fig:intZ}
we see that on $Z$ there are $8$ disjoint
$(-1)$-curves, namely 
$\bar E_1,\dots,\bar E_4,
\bar E_{(12)(34)},\bar E_{(13)(24)},
\bar E_{(23)(14)},\bar E$. If we contract
all of them, $\bar E'$ becomes a $(-1)$-curve as well
and if we contract also this one, we 
obtain a birational map $Z\to\pp^2$.
We denote by $q_i,\ q_{(ij)(kl)}$ and 
$q$ the images of $\bar E_i,
\ \bar E_{(ij)(kl)}$ and $\bar E'$
respectively.
The image of $\bar E_{12}$
is a line $l_{12}$, passing through $q_1, 
q_2$ and $q_{(12)(34)}$, and anagolously
$l_{13}$ and $l_{23}$.
The image of $\bar \Delta_-$ is a conic
passing through $q_1,\dots,
q_4$ and $q$. In Figure~\ref{fig:plane}
we represent the configuration of 
the points that we are blowing up
on $\pp^2$ (together with the tangent 
direction to the conic at $q$)
in order to obtain $Z$.
\begin{figure}

\begin{center}
\begin{tikzpicture}[scale=0.4,
main node/.style={circle,fill=red!70,draw,minimum size=5pt,inner sep=1pt},
0 node/.style={circle,draw,minimum size=5pt,inner sep=1pt},
1 node/.style={circle,fill=blue!70,draw,minimum size=5pt,inner sep=1pt}]

\node (1) at (7,8) {};
\node (2) at (9,8) {};
\node (3) at (-1,-8) {};
\node (4) at (1,-8) {};
\node (5) at (-6,-2) {};
\node (6) at (6,-2) {};
\node (7) at (-9,8) {};
\node (8) at (-7,8) {};
\node (9) at (0,8) {};
\node (10) at (7,0) {};
\node (11) at (-7,0) {};
\node (12) at (0,-8) {};

\path[every node/.style={font=\sffamily\small}]
(1) edge (3)
(2) edge (5)
(4) edge (8)
(6) edge (7)
(10) edge (11)
(9) edge (12)
;

\draw(0,-2) ellipse (3.5 and 4);

\node[1 node] (20) at (0,2) {}
node[above left] at (-.2,1.8) {\tiny $q_1$};

\node[1 node] (21) at (-3,0) {}
node[above left] at (-3,0) {\tiny $q_2$};

\node[1 node] (23) at (0,-6) {}
node[below right] at (0,-6) {\tiny $q_3$};

\node[1 node] (22) at (3,0) {}
node[above right] at (3,0) {\tiny $q_4$};

\node[1 node] (20) at (6,6) {}
node[below right] at (6,6) {\tiny $q_{(12)(34)}$};

\node[1 node] (21) at (-6,6) {}
node[below left] at (-6,6) {\tiny $q_{(23)(14)}$};

\node[1 node] (22) at (0,0) {}
node[below right] at (0,0) {\tiny $q_{(13)(24)}$};

\node[1 node] (24) at (3,-4) {}
node[below right] at (3,-4) {\tiny $q$};


\end{tikzpicture}
\end{center}

\caption{Planar configuration}
\label{fig:plane}
\end{figure}
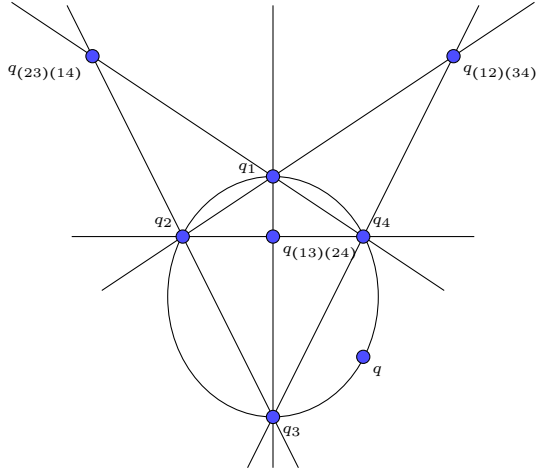

Therefore, for any $n > 0$, the curve $D_n$
appearing in
Proposition~\ref{prop:dn} corresponds to
a plane curve. We can compute that its degree
is $3n^2+n$, intersecting $R_n = \pi^*D_n-
\frac{1}{2}\bar E'$ (see the proof of
Proposition~\ref{prop:dn}) with the pull-back
of $E_{12}$, i.e. the class 
$\bar E_{12} + \bar E_1 +
\bar E_2 + \bar E_{(12)(34)}$ in $Z$.
In the same way we see that the multiplicity
in $q_1$ is $n^2$, by taking the intersection
with the pull-back of $E_1$, namely $\bar E_1
+ 1/2 (\bar E_{12} + \bar E_{13} + 
\bar E_{14})$, and the same holds for 
$q_2,q_3,q_4$. Computing the 
intersection with $\bar E_{(12)(34)} + 
1/2(\bar E_{12} + \bar E_{34})$ we have
that the multiplicity in
$q_{(12)(34)}$ is $n^2+n$, and the
same holds for $q_{(13)(24)}$ and 
$q_{(23)(14)}$. Finally, the 
curve has multiplicity $n^2$ at $q$,
and multiplicity $n^2-1$ at the point infinitely 
near to $q$, in the direction of the conic.

\end{remark}


\end{document}